\documentclass[12pt]{paper}

\usepackage{color}
\usepackage{mathrsfs}
\usepackage{mathptmx}       
\usepackage{helvet}         
\usepackage{courier}        
\usepackage{type1cm}        
\usepackage{makeidx}         
\usepackage{graphicx}        
\usepackage{multicol}        
\usepackage[bottom]{footmisc}
\usepackage{amsmath}
\usepackage{amsfonts}
\usepackage{amssymb}
\usepackage{amsthm}
\usepackage{amsopn}
\usepackage{subcaption}
\usepackage{comment}
\newtheorem{theorem}{Theorem}

\newtheorem{lemma}[theorem]{Lemma}

\newtheorem{remark}[theorem]{Remark}

\usepackage[top=2cm, bottom=2cm, left=2cm, right=2cm]{geometry}
\numberwithin{theorem}{section}
\numberwithin{figure}{section}
\numberwithin{equation}{section}

\DeclareMathOperator{\dist}{dist}
\DeclareMathOperator{\SLE}{SLE}

\DeclareMathOperator{\dimH}{dim}

\begin{document}

\title{Remarks on the intersection of  $\SLE_{\kappa}(\rho)$ curve with the real line}
\author{Menglu WANG and Hao WU}


%
%
\maketitle

\abstract{$\SLE_{\kappa}(\rho)$ is a variant of $\SLE_{\kappa}$ where $\rho$ characterizes the repulsion (if $\rho>0$) or attraction $(\rho<0)$ from the boundary.
This paper examines the probabilities of $\SLE_{\kappa}(\rho)$ to get close to the boundary. We show how close the chordal $\SLE_{\kappa}(\rho)$ curves get to the boundary asymptotically, and provide an estimate for the probability that the $\SLE_{\kappa}(\rho)$ curve hits graph of functions.
These generalize the similar result derived by Schramm and Zhou for standard $\SLE_{\kappa}$. }
\medbreak
\medbreak
\newcommand{\eps}{\epsilon}
\newcommand{\ov}{\overline}
\newcommand{\U}{\mathbb{U}}
\newcommand{\T}{\mathbb{T}}
\newcommand{\HH}{\mathbb{H}}
\newcommand{\LA}{\mathcal{A}}
\newcommand{\LC}{\mathcal{C}}
\newcommand{\LF}{\mathcal{F}}
\newcommand{\LK}{\mathcal{K}}
\newcommand{\LE}{\mathcal{E}}
\newcommand{\LL}{\mathcal{L}}
\newcommand{\LU}{\mathcal{U}}
\newcommand{\LV}{\mathcal{V}}
\newcommand{\LZ}{\mathcal{Z}}
\newcommand{\R}{\mathbb{R}}
\newcommand{\C}{\mathbb{C}}

\newcommand{\N}{\mathbb{N}}
\newcommand{\Z}{\mathbb{Z}}
\newcommand{\E}{\mathbb{E}}
\newcommand{\PP}{\mathbb{P}}
\newcommand{\QQ}{\mathbb{Q}}
\newcommand{\A}{\mathbb{A}}
\newcommand{\bn}{\mathbf{n}}
\newcommand{\MR}{MR}
\newcommand{\cond}{\,|\,}
\newcommand{\la}{\langle}
\newcommand{\ra}{\rangle}
\newcommand{\tree}{\Upsilon}
\section{Introduction}
Schramm Loewner Evolution ($\SLE$) is a random fractal curve in a simply connected domain connecting two boundary points. It is introduced by Oded Schramm \cite{SchrammFirstSLE} as the candidates of the scaling limits of discrete statistical models. It is indexed by a nonnegative real $\kappa\ge 0$. When $\kappa\in[0,4]$, $\SLE_{\kappa}$ are continuous simple curves and they do not touch the boundary except at the end points; when $\kappa>4$, they are still continuous curves but they touch the boundary and also touch themselves. We focus on the intersection of $\SLE$ paths (in the upper half-plane from the origin to $\infty$) with the boundary. When $\kappa>4$, the curves touch the boundary and the Hausdorff dimension of the intersection is $(2-8/\kappa)\wedge 1$ which is derived by Alberts and Sheffield in \cite{AlbertsSheffieldDimension}; when $\kappa\in [0,4]$, the curves do not touch the boundary,  Shramm and Zhou examined how close they get to the boundary asymptotically far away from the starting point in \cite{SchrammZhouBoundarySLE}.

$\SLE_{\kappa}(\rho)$ is a variant of $\SLE_{\kappa}$ process. Roughly speaking, the parameter $\rho$ tells that there is an attraction ($\rho<0$) or repulsion ($\rho>0$) from the boundary. When $\rho<\kappa/2-2$, the corresponding $\SLE_{\kappa}(\rho)$ curves will touch the boundary; when $\rho\ge \kappa/2-2$, the curves do not touch the boundary except at the end points. In this paper, we generalize the conclusion on the Hausdorff dimension result in \cite{AlbertsSheffieldDimension} to the intersection of $\SLE_{\kappa}(\rho)$ process with the boundary for $\rho<\kappa/2-2$ and generalize the conclusion on the boundary proximity result in \cite{SchrammZhouBoundarySLE} to $\SLE_{\kappa}(\rho)$ process for $\rho\ge \kappa/2-2$.

Note that the Hausdorff dimension for the intersection of $\SLE_{\kappa}(\rho)$ with the real line has been derived in \cite{MillerWuSLEIntersection} with a weaker version of One-Point Estimate and Two-Point Estimate where there are error terms in the exponent. In our paper, we derive these estimates up to constant.
In \cite{SchrammZhouBoundarySLE}, the authors give the boundary proximity result as well as a precise two-point Green's function on the boundary. In our paper, we only generalize the boundary proximity result to $\SLE_{\kappa}(\rho)$ process, but we do not get the corresponding Green's function for $\SLE_{\kappa}(\rho)$.

The following quantity is special and we fix it throughout the paper:

\begin{equation}\label{eqn::exponent}
\alpha=\alpha(\kappa, \rho):=\frac{(\rho+2)(\rho+4-\kappa/2)}{\kappa}.
\end{equation}

Note that, $\alpha>0$ when $\rho>(-2)\vee(\kappa/2-4)$; and $\alpha\ge 1$ when $\rho\ge \kappa/2-2$.

\begin{theorem}\label{thm::boundary_dimension}
Fix $\kappa>0$ and $\rho\in ((\kappa/2-4)\vee (-2), \kappa/2-2)$. Let $\eta$ be an $\SLE_{\kappa}(\rho)$ process with a single force point located at $0^+$. Almost surely,
\[\dimH_{H}(\eta\cap\R_+)=1-\alpha.\]
\end{theorem}

Fix $r>0$, for a function $h:[r,\infty)\rightarrow(0,\infty)$, we denote the region under its graph by $\Gamma^{h}:=\{x+iy: x\in[r,\infty), 0< y\le h(x)\}$, and define
\begin{equation}\label{eqn::lambda_graph}
\Lambda^{h}_{\kappa,\rho}:=\begin{cases}
h(x)^{\alpha-1} &\rho>\frac{\kappa}{2}-2 ,\\
1/\log(\frac{x}{h(x)}\vee 2)&\rho=\frac{\kappa}{2}-2.
\end{cases}\end{equation}

\begin{theorem}\label{thm::boundary_proximity}
Fix $\kappa>0$ and $\rho\ge\frac{\kappa}{2}-2$. Let $\eta$ be an $\SLE_{\kappa}(\rho)$ process with a single force point located at $0^+$. Fix $r>1$, and suppose that $h:[r,\infty)\rightarrow(0,\infty)$ is continuous and satisfies
\begin{equation}\label{condition1}
\sup\{\Lambda^{h}_{\kappa,\rho}(x)/\Lambda^{h}_{\kappa,\rho}(y): r\le x\le y\le 2x\}<\infty.
\end{equation}
If
\begin{equation}\label{condition2}
\int_{r}^{\infty}\frac{\Lambda^{h}_{\kappa,\rho}(x)}{x^{\alpha}}\textrm{d}x<\infty,
\end{equation}
then $\eta\cap\Gamma^{h}$ is bounded a.s. Conversely, if the integral in Equation (\ref{condition2}) is infinite, then $\eta\cap\Gamma^{h}$ is unbounded a.s.
\end{theorem}
\begin{remark}
As a consequence of Theorem \ref{thm::boundary_proximity}, we have the following observations.

If $\rho>\kappa/2-2$ and \[h(x)=x(\log x)^{-u},\] then $\eta\cap\Gamma^{h}$ is a.s. bounded when  $u>1/(\alpha-1)$, and a.s. unbounded when $u\le 1/(\alpha-1)$.

If $\rho=\kappa/2-2$ and \[h(x)=x^{-(\log\log x)^{u}},\] then $\eta\cap\Gamma^{h}$ is a.s. bounded when  $u>1$, and a.s. unbounded when $u\le 1$.
\end{remark}

\begin{remark} \label{rem::estimate_hittingproba}
Assume the same notations as in Theorem \ref{thm::boundary_proximity}.
We can also obtain an estimate on the probability that $\eta$ hits $\Gamma^h$. Suppose that Equations (\ref{condition1}) and (\ref{condition2}) hold and suppose further that $h(x)\le x/2$ for all $x\ge r$, then we have
\[\PP[\eta\cap \Gamma^h\neq\emptyset]\asymp 1\wedge \int_{r}^{\infty}\frac{\Lambda^{h}_{\kappa,\rho}(x)}{x^{\alpha}}\textrm{d}x.\]
\end{remark}

The most important ingredient for the proofs of Theorems \ref{thm::boundary_dimension} and \ref{thm::boundary_proximity} is the following estimate.\footnote{We write $f\asymp g$ if there exists a constant $C\ge 1$ such that $C^{-1}f(x)\le g(x)\le Cf(x)$ for all $x$. We write $f\lesssim g$ if there exists a constant $C>0$ such that $f(x)\le Cg(x)$ and $f\gtrsim g$ if $g\lesssim f$. We denote by $\dist$ the Euclidean distance.}
\begin{theorem}\label{thm::assumption}
Fix $\kappa>0$ and $\rho>(\kappa/2-4)\vee(-2)$. Let $\eta$ be an $\SLE_{\kappa}(\rho)$ process with a single force point located at $x^{R}\in[0^+,1/2]$. For all $0<\epsilon<1$, $\delta>0$, $x>0$, we have One-Point Estimate
\begin{equation}\label{onepoint}
\PP\left[\dist(1, \eta)\le \eps\right]\asymp\epsilon^{\alpha},
\end{equation}
and Two-Point Estimate
\begin{equation}\label{twopoints}
\PP\left[\dist(1, \eta)\le \eps, \dist(1+x, \eta)\le \delta\right]\lesssim\epsilon^{\alpha}\delta^{\alpha}x^{-\alpha},
\end{equation}
where $\alpha$ is the same as in Equation (\ref{eqn::exponent}) and the constants in $\asymp$ and in $\lesssim$ are uniform over all $x^{R}\in[0^+,1/2]$ and $x>0$.
\end{theorem}

We conclude the introduction by explaining the relation to previous works.
Theorem \ref{thm::boundary_dimension} has been derived in \cite[Theorem 1.6]{MillerWuSLEIntersection} by the coupling between $\SLE$ paths and Gaussian Free Field. Theorem \ref{thm::assumption} has been obtained in \cite{AlbertsKozdronIntersectionProbaSLEBoundary} for standard $\SLE_{\kappa}$ curves, and it has been obtained in \cite[Proposition 7]{WernerWuCLEtoSLE} for $\SLE_{\kappa}(\rho)$ with $\kappa\in [8/3,4]$ by the construction of $\SLE$ paths from Brownian excursion soup and Brownian loop soup. In our paper, in the proof of Theorem \ref{thm::assumption}, we use tools developed in a recent work by Greg Lawler \cite{LawlerMinkowskiSLERealLine}, and we prove it for all $\SLE_{\kappa}(\rho)$ processes.
\medbreak
\noindent\textbf{Outline.}We will give preliminaries of $\SLE_{\kappa}(\rho)$ process in Section \ref{sec::preliminaries}. We prove Theorem \ref{thm::assumption} in Section \ref{sec::onepoint}. In fact, once we obtain One-Point Estimate and Two-Point Estimate (and their conditional versions), the proofs of Theorems \ref{thm::boundary_dimension} and \ref{thm::boundary_proximity} become standard. We will briefly explain these proofs in Section \ref{sec::boundaryproximity}.

\medbreak
\noindent\textbf{Acknowledgment.} The authors thank Scott Sheffield for helpful comments on the previous version of this article. H. Wu' s work is funded by NSF DMS-1406411.

\section{Preliminaries}\label{sec::preliminaries}
Loewner chain is a collection of compact hulls $(K_{t}, t\ge 0)$ associated with the family of conformal maps $(g_{t}, t\ge 0)$ obtained by solving the Loewner equation: for each $z\in\mathbb{H}$,
\begin{equation}\label{loewner}
\partial_{t}{g}_{t}(z)=\frac{2}{g_{t}(z)-W_{t}}, \quad g_{0}(z)=z,
\end{equation}
where $(W_t, t\ge 0)$ is a one-dimensional continuous function which we call the driving function. Let $T_z$ be the swallowing time of $z$ defined as $\sup\{t\ge 0: \min_{s\in[0,t]}|g_{s}(z)-W_{s}|>0\}$.
Let $K_{t}:=\overline{\{z\in\mathbb{H}: T_{z}\le t\}}$. Then $g_{t}$ is the unique conformal map from $H_{t}:=\mathbb{H}\backslash K_{t}$ onto $\mathbb{H}$ such that $\lim_{|z|\rightarrow\infty}|g_{t}(z)-z|=0$.

An $\SLE_{\kappa}$ is the random Loewner chain $(K_{t}, t\ge 0)$ driven by $W_t=\sqrt{\kappa}B_t$ where $(B_t, t\ge 0)$ is a standard one-dimensional Brownian motion.
In \cite{RohdeSchrammSLEBasicProperty}, the authors prove that $(K_{t}, t\ge 0)$ is almost surely generated by a continuous curve, i.e. there almost surely exists a continuous curve $\eta$ such that for each $t\ge 0$, $H_{t}$ is the unbounded connected component of $\mathbb{H}\backslash\eta([0,t])$.

\subsection{$\SLE_{\kappa}(\rho)$}

$\SLE_{\kappa}(\rho)$ process is a generalization of $\SLE_{\kappa}$ in which one keeps track of one additional marked point which we call the force point. Suppose that $x^{R}\ge 0$. We associate with the force point $x^{R}$ a weight $\rho\in\mathbb{R}$. An SLE$_{\kappa}(\rho)$ process with force point $x^{R}$ is the Loewner chain driven by $W_{t}$ which is the solution to the following systems of SDEs:
\begin{equation}\label{Loewnereq}
\begin{gathered}
dW_{t}=\sqrt{\kappa}dB_{t}+\frac{\rho dt}{W_{t}-V_{t}}, \quad W_0=0;\quad dV_{t}=\frac{2dt}{V_{t}-W_{t}}, \quad V_{0}=x^{R}.\\
\end{gathered}
\end{equation}
We also say that $\SLE_{\kappa}(\rho)$ process is driven by the pair $(W_t, V_t, t\ge 0)$.
Fix $\kappa>0,\rho>-2$, the solution to Equation (\ref{Loewnereq}) exists for all times $t\ge 0$. The corresponding Loewner chain is almost surely generated by a continuous curve $\eta$ which is almost surely transient (\cite[Section 8]{LawlerSchrammWernerConformalRestriction} and \cite[Theorem 1.3]{MillerSheffieldIG1}): $\lim_{t\rightarrow\infty}\eta(t)=\infty$. 
If $\rho\ge\kappa/2-2$, the curve $\eta$ almost surely does not hit the positive real line $\mathbb{R}_{+}$.

$\SLE_{\kappa}(\rho)$ process satisfies Scaling Invariance and Domain Markov Property: 
Fix $\kappa>0,\rho>-2$, let $(K_t, t\ge 0)$ be the chordal Loewner chain corresponding to $\SLE_{\kappa}(\rho)$ process with force point $x^R$, and let $(W_t, V_t, t\ge 0)$ be the driving function.
\begin{itemize}
\item (Scaling Invariance) For any $\lambda>0$, ($\lambda^{-1}K_{\lambda^{2}t}, t\ge 0$) has the same law as an $\SLE_{\kappa}(\rho)$ process with force point $x^{R}/\lambda$. In particular, if $x^{R}=0^+$, it is scaling invariant.
\item (Domain Markov Property) Define $f_{t}:=g_{t}-W_{t}$ as the centered conformal map, then for any finite stopping time $\tau$, the curve $(\tilde{\eta}(t):=f_{\tau}(\eta(t+\tau)), t\ge 0)$ is an $\SLE_{\kappa}(\rho)$ process with force point $V_{\tau}-W_{\tau}$.
\end{itemize}

\begin{lemma}\label{lem::martingale}
Fix $\kappa>0, \rho>-2$ and $\alpha$ is defined in Equation (\ref{eqn::exponent}).
Suppose that $\eta$ is an $\SLE_{\kappa}(\rho)$ process with a single force point located at $x^{R}\ge 0$. Let $(g_{t}, t\ge 0)$ be the corresponding conformal maps driven by $(W_{t}, V_{t}, t\ge 0)$. Define for $x>x^{R}$, $t\ge 0$,
\begin{equation}\label{eqn::martingale}
M_{t}^{x}:=\left(\frac{g_{t}'(x)}{g_{t}(x)-V_{t}}\right)^{\alpha}\left(\frac{g_{t}(x)-V_{t}}{g_{t}(x)-W_{t}}\right)^{\beta}.
\end{equation}
where $\beta=(8+2\rho-\kappa)/\kappa$. Then $M_{t}^{x}$ is well-defined up to the first time that $x$ is swallowed and $(M_t^x, t\ge 0)$ is a local martingale.
\end{lemma}
\begin{proof}
\cite[Theorem 6 and Remark 7]{SchrammWilsonSLECoordinatechanges}.
\end{proof}
\begin{lemma}\label{lem::mart_dist}
Assume the same notations as in Lemma \ref{lem::martingale}. For $x>x^R$ and $t>0$ which is before the swallowing time of $x$, we have that
\[\frac{x-x^R}{4}\dist(x, \eta([0,t]))\le \frac{g_t(x)-V_t}{g_t'(x)}\le 4\dist(x, \eta([0,t])).\]
\end{lemma}

\begin{proof}
We may assume that $x=1$.

Suppose that $(K_t, t\ge 0)$ is the chordal Loewner chain corresponding to $\SLE_{\kappa}(\rho)$ process with force point $x^R$ and that $(W_t, V_t, t\ge 0)$ is the driving function.
Let $\overline{K}_{t}:=\{\bar{z}:z\in K_{t}\}$, and by Schwarz Reflection, the conformal map $g_{t}$ can be extended to $\mathbb{C}\backslash(K_{t}\cup\overline{K}_{t}\cup\mathbb{R}_{-})$. Let $O_t$ be the image of the rightmost point of $K_t\cap\R$ under $g_t$. Then Koebe 1/4 Theorem implies that
\begin{equation*}
\frac{1}{4}\dist\left(1, \eta([0,t])\right)\le \frac{g_{t}(1)-O_{t}}{g'_{t}(1)}\le 4\dist\left(1, \eta([0,t])\right).
\end{equation*}
Thus,
\[\frac{g_{t}(1)-V_t}{g'_{t}(1)}\le \frac{g_{t}(1)-O_{t}}{g'_{t}(1)}\le 4\dist\left(1, \eta([0,t])\right).\]
For the lower bound, it is sufficient to show that
\[g_t(1)-O_t\le \frac{1}{1-x^R}(g_t(1)-V_t).\]
If $x^R\in K_t$, then $O_t=V_t$ and it holds. If $x^R\not\in K_t$, let $u$ be a real between the rightmost point of $K_t\cap\R$ and $x^R$, then, there exists some $\tilde{u}\in [u, x^R]$ such that
\[V_t-g_t(u)=g_t(x^R)-g_t(u)=g_t'(\tilde{u})(x^R-u).\]
There exists some $\tilde{v}\in [x^R,1]$ such that
\[g_t(1)-V_t=g_t(1)-g_t(x^R)=g_t'(\tilde{v})(1-x^R).\]
Note that $g_t'(y)\in [0,1]$ and $y\mapsto g_t'(y)$ is increasing, thus
\[V_t-g_t(u)\le g_t'(\tilde{u})x^R\le g_t'(\tilde{v})x^R=\frac{x^R}{1-x^R}(g_t(1)-V_t).\]
Let $u$ approach the rightmost point of $K_t\cap\R$, we have that
\[V_t-O_t\le \frac{x^R}{1-x^R}(g_t(1)-V_t).\]
This implies that, we always have
\[g_t(1)-O_t\le \frac{1}{1-x^R}(g_t(1)-V_t),\]
as desired.
\end{proof}

In \cite{LawlerMinkowskiSLERealLine}, the author has proved an estimate for $\SLE_{\kappa}(\rho)$ process which is closely related to Equation (\ref{onepoint}). We rephrase that result using our notations in the present setting.

\begin{lemma}\label{lawler}
Assume the same notations as in Lemma \ref{lem::martingale} with $\rho>(\kappa/2-4)\vee(-2)$ and $x^{R}\in[0^+, 1)$.
For $0<\epsilon<1/2$, we define
\begin{equation*}
\sigma(\epsilon):=\inf\left\{t\ge 0: \frac{g_{t}(1)-V_t}{g'_{t}(1)}\le\epsilon(1-x^{R})\right\},
\end{equation*}
then we have
\begin{equation}\label{conformalest}
\mathbb{P}[\sigma(\epsilon)<\infty]\asymp\epsilon^{\alpha}(1-x^{R})^{\beta},
\end{equation}
where $\alpha$ is the same as in Equation (\ref{eqn::exponent}) and $\beta>0$ is the same as in Lemma \ref{lem::martingale}, and the constants in $\asymp$ are uniform over all $x^{R}\in[0^+, 1)$. \end{lemma}
\begin{proof}
\cite[Proposition 5.4]{LawlerMinkowskiSLERealLine}. Although \cite[Proposition 5.4]{LawlerMinkowskiSLERealLine} states this conclusion for $\kappa\in (0,8)$, the same calculation also works for $\kappa\ge 8$.
\end{proof}

\subsection{$\SLE_{\kappa}(\rho^L;\rho^R)$}
Analogously, we can define an SLE$_{\kappa}(\rho^{L}; \rho^{R})$ process with two force points $(x^{L}; x^{R})$ where $x^{L}\le 0\le x^{R}$. It is the Loewner chain driven by $W_{t}$ which is the solution to the following systems of SDEs:
\[dW_{t}=\sqrt{\kappa}dB_{t}+\frac{\rho^{L} dt}{W_{t}-V_{t}^{L}}+\frac{\rho^{R} dt}{W_{t}-V_{t}^{R}}, \quad W_{0}=0;\]
\[dV^{L}_{t}=\frac{2dt}{V^{L}_{t}-W_{t}}, \quad V^{L}_{0}=x^{L}; \quad dV^{R}_{t}=\frac{2dt}{V^{R}_{t}-W_{t}}, \quad V^{R}_{0}=x^{R}.\]

When $\rho^L>-2$ and $\rho^R>-2$, the solution exists for all times $t\ge 0$, and the corresponding Loewner chain is almost surely generated by a continuous curve which is almost surely transient (\cite[Section 2]{MillerSheffieldIG1}). 

\begin{lemma}\label{lem::sle_continuity_forcepoints}
Fix $\kappa>0, \rho^L>-2, \rho^R>-2$. Suppose that $(x^L_n)$ (resp. $(x^R_n)$) is a sequence of negative (resp. positive) real numbers converging to $x^L\le 0^-$ (resp. $x^R\ge 0^+$) as $n\to\infty$. For each $n$, suppose that $(W^n, V^{n,L}, V^{n,R})$ is the driving triple for $\SLE_{\kappa}(\rho^L;\rho^R)$ with force points $(x^L_n; x^R_n)$. Then $(W^n, V^{n,L}, V^{n,R})$ converges weakly in law with respect to the local uniform topology to the driving triple $(W, V^L, V^R)$ of $\SLE_{\kappa}(\rho^L;\rho^R)$ with force points $(x^L;x^R)$ as $n\to\infty$.
\end{lemma}
\begin{proof}
\cite[Section 4.7]{LawlerConformallyInvariantProcesses}.
\end{proof}

\begin{lemma}\label{lem::lawler_generalized}
Fix $\kappa>0, \rho^L>-2, \rho^R>(\kappa/2-4)\vee(-2)$. Let $\eta$ be an $\SLE_{\kappa}(\rho^L;\rho^R)$ process with force points $(x^L; x^R)$ where $x^L\le 0, x^R\in [0^+, 1)$. Let $(g_t, t\ge 0)$ be the family of conformal maps and $(W_t, V^L_t, V^R_t, t\ge 0)$ be the driving function.  
For $0<\epsilon<1/2$, we define
\begin{equation*}
\sigma(\epsilon):=\inf\left\{t\ge 0: \frac{g_{t}(1)-V^R_t}{g'_{t}(1)}\le\epsilon(1-x^{R})\right\},
\end{equation*}
then we have
\begin{equation}\label{conformalest}
\mathbb{P}[\sigma(\epsilon)<\infty]\asymp\epsilon^{\alpha}(1-x^{R})^{\beta},
\end{equation}
where $\alpha$ is the same as in Equation (\ref{eqn::exponent}) with $\rho=\rho^R$ and $\beta>0$ is the same as in Lemma \ref{lem::martingale} with $\rho=\rho^R$, and the constants in $\asymp$ are uniform over all $x^{R}\in[0^+, 1)$, $x^{L}\le 0$. \end{lemma}

\begin{proof}
By Lemma \ref{lem::sle_continuity_forcepoints}, we know that the probability $\PP[\sigma(\eps)<\infty]$ is continuous in $x^L$. 
Combining with Lemma \ref{lawler}, we obtain the conclusion. 
\end{proof}

\section{Proof of Theorem \ref{thm::assumption}}\label{sec::onepoint}
In this section, we fix $\kappa>0$ and $\rho>(\kappa/2-4)\vee(-2)$. Let $\eta$ be an $\SLE_{\kappa}(\rho)$ process with a single force point located at $x^{R}\in[0^+,1/2]$. Let $(g_t, t\ge 0)$ be the family of conformal maps and $(W_t, V_t, t\ge 0)$ be the driving function. Define $f_t:=g_t-W_t$ to be the centered conformal map.
\begin{proof} [Proof of Equation (\ref{onepoint})]
From Lemma \ref{lem::mart_dist}, we have that
\[\frac{1-x^R}{4}\dist\left(1, \eta([0,t])\right)\le \frac{g_{t}(1)-V_t}{g'_{t}(1)}\le 4 \dist\left(1, \eta([0,t])\right).\]
Without loss of generality, we may assume $0<\epsilon<1/16$. By Equation (\ref{lawler}), we have, uniform over $x^R\in [0^+, 1/2]$,
\[
\PP\left[\dist(1, \eta)\le \eps\right]\le \PP[\sigma(4\epsilon/(1-x^R))<\infty]\asymp\epsilon^{\alpha},\quad 
\PP\left[\dist(1, \eta)\le \eps\right]\ge \PP[\sigma(\eps/4)<\infty]\asymp \eps^{\alpha}.\]
\end{proof}

We can also prove a conditional version of One-Point Estimate.
\begin{lemma}\label{lem::onepoint_conditional}
Suppose that $T$ is any stopping time with $\dist(1, \eta([0,T]))\ge 2\eps$, then we have the upper bound
\begin{equation}\label{eqn::one_conditional_upper}
\PP[\dist(1, \eta)\le \eps\cond \eta([0,T])]\lesssim \left(\frac{\eps}{\dist(1, \eta([0,T]))}\right)^{\alpha},\end{equation}
where the constant in $\lesssim$ is uniform over $x^R\in [0^+,1/2]$ and is independent of $T$.

Moreover, if we assume that $V_T-W_T\le g_T(1)-V_T$, we also have the lower bound
\begin{equation}
\PP[\dist(1, \eta)\le \eps\cond \eta([0,T])]\gtrsim \left(\frac{\eps}{\dist(1, \eta([0,T]))}\right)^{\alpha},
\end{equation}
where the constant in $\gtrsim$ is uniform over $x^R\in [0^+,1/2]$ and is independent of $T$ as long as $V_T-W_T\le g_T(1)-V_T$.
\end{lemma}

\begin{proof}
Let $O_t$ be the image of the rightmost point of $K_t\cap \R$ under $g_t$.
Without loss of generality, we may assume that
\[\frac{\epsilon}{\dist\left(1,\eta([0,T])\right)}\le 2^{-8}.\]
By Koebe 1/4 Theorem, we have
\[\frac{g_{T}(1)-O_{T}}{g'_{T}(1)}\ge
\frac{\dist(1, \eta([0,T]))}{4}\ge 4\epsilon, \]
so that, $f_T^{-1}$ is well-defined in the ball with center $f_T(1)$ and radius $4\eps g_T'(1)$, and Koebe 1/4 Theorem implies that the image of the ball $B(1,\eps)$ under $f_T$ is contained in the ball with center $f_T(1)$ and radius $4\eps g_T'(1)$.
Apply Koebe 1/4 Theorem to $f_T$, we have that the image of the ball $B(1,\eps)$ under $f_T$ contains the ball with center $f_T(1)$ and radius $\eps g_T'(1)/4$.
Define, for $t\ge 0$,
\[\tilde{\eta}(t):=f_T(\eta(t+T))/f_T(1).\]
Then, Domain Markov Property implies that $\tilde{\eta}$ has the same law as $\SLE_{\kappa}(\rho)$ with force point located at $(V_T-W_T)/f_T(1)$. Given $\eta([0,T])$, we have that
\[\left[\dist(1,\tilde{\eta})\le \frac{\eps g_T'(1)}{4f_T(1)}\right]\subset [\dist(1,\eta)\le \eps]\subset \left[\dist(1,\tilde{\eta})\le \frac{4\eps g_T'(1)}{f_T(1)}\right].\]
Let $\tilde{\sigma}$ be the stopping time for $\tilde{\eta}$ that is defined in the same way as in Lemma \ref{lawler}.
Then we have that
\begin{align*}
\PP[\dist(1,\eta)\le\eps\cond \eta([0,T])]&\le \PP\left[\dist(1,\tilde{\eta})\le \frac{4\eps g_T'(1)}{f_T(1)}\cond \eta([0,T])\right]\\
&\le \PP\left[\tilde{\sigma}\left(\frac{16\eps g_T'(1)}{g_T(1)-V_T}\right)<\infty\cond \eta([0,T])\right]\tag{\text{Apply Lemma \ref{lem::mart_dist} to $\tilde{\eta}$}}\\
&\asymp \left(\frac{\eps g_T'(1)}{g_T(1)-V_T}\right)^{\alpha}\left(\frac{g_T(1)-V_T}{f_T(1)}\right)^{\beta}\tag{\text{Apply Lemma \ref{lawler} to $\tilde{\eta}$}}\\
&\le \left(\frac{\eps g_T'(1)}{g_T(1)-V_T}\right)^{\alpha},
\end{align*}
\begin{align*}
\PP[\dist(1,\eta)\le\eps\cond \eta([0,T])]&\ge \PP\left[\dist(1,\tilde{\eta})\le \frac{\eps g_T'(1)}{4f_T(1)}\cond \eta([0,T])\right]\\
&\ge \PP\left[\tilde{\sigma}\left(\frac{\eps g_T'(1)}{16f_T(1)}\right)<\infty\cond \eta([0,T])\right]\tag{\text{Apply Lemma \ref{lem::mart_dist} to $\tilde{\eta}$}}\\
&\asymp \left(\frac{\eps g_T'(1)}{f_T(1)}\right)^{\alpha}\left(\frac{g_T(1)-V_T}{f_T(1)}\right)^{\beta}\tag{\text{Apply Lemma \ref{lawler} to $\tilde{\eta}$}}\\
&= \left(\frac{\eps g_T'(1)}{g_T(1)-V_T}\right)^{\alpha}\left(\frac{g_T(1)-V_T}{f_T(1)}\right)^{\beta+\alpha},
\end{align*}
where the constants in $\asymp$ are uniform over $x^R\in [0^+, 1)$ and are independent of $T$.
From Lemma \ref{lem::mart_dist}, we have that
\[\dist(1,\eta([0,T]))\asymp \frac{g_T(1)-V_T}{g_T'(1)},\]
where the constants in $\asymp$ are uniform over $x^R\in [0^+, 1/2]$. This implies the upper bound.

For the lower bound, under the assumption that $V_T-W_T\le g_T(1)-V_T$, we have
\[\frac{1}{2}\le \frac{g_T(1)-V_T}{f_T(1)}\le 1,\]
which implies the lower bound.
\end{proof}

\begin{lemma}\label{lem::rare_event}
For $\epsilon, \delta, x>0$, we define stopping times
\[T:=\inf\{t\ge 0: \dist(1, \eta([0,t]))\le \eps\},\quad S:=\inf\{t\ge 0: \dist(1+x, \eta([0,t]))\le \delta\}.\]
Define
\[\tilde{\alpha}:=\alpha(\kappa, \rho+2)=\frac{(\rho+4)(\rho+6-\kappa/2)}{\kappa}.\]
We have that
\[\PP[S<T<\infty]\lesssim \eps^{\alpha}\delta^{\tilde{\alpha}}x^{-\tilde{\alpha}}\]
where the constant in $\lesssim$ is uniform over $x^R\in [0^+,1/2]$ and $x>0$.
\end{lemma}

\begin{proof}
Recall that $\alpha$ is defined in Equation (\ref{eqn::exponent}) and $\beta$ is defined in Lemma \ref{lem::martingale}. Define
\[M_t:=\left(\frac{g_t'(1)}{g_t(1)-V_t}\right)^{\alpha}\left(\frac{g_t(1)-V_t}{g_t(1)-W_t}\right)^{\beta}.\]

By \cite[Theorem 6]{SchrammWilsonSLECoordinatechanges}, we know that $(M_t)_{t\ge 0}$ is a local martingale. Let $\phi(z)=xz/(1+x)(1-z)$ be the M\"{o}bius transformation of the upper half plane such that it maps the triple $(0,1,1+x)$ to $(0,\infty,-1)$. Let $\eta^*$ be an $\SLE_{\kappa}(\rho+2;\rho)$ process with force points $(\phi(\infty);\phi(x^R))$, and we denote by $\PP^*$ its law. From \cite{SchrammWilsonSLECoordinatechanges}, we also know that the law of the image of $\eta$ under $\phi$ weighted by $M$ is the same as $\PP^*$. Thus, we have
\begin{align*}
\PP[S<T<\infty]&=\E\left[1_{[S<T]}\PP[T<\infty\cond \eta([0,S])]\right]\\
&\lesssim\eps^{\alpha}\E\left[M_S 1_{[S<T]}\right]\tag{\text{By Lemmas \ref{lem::mart_dist} and \ref{lawler}}}\\
&\le \eps^{\alpha}M_0\PP^*\left[\dist(-1,\eta^*)\le \frac{4\delta}{(1+x)x}\right]\\
&\asymp \eps^{\alpha}\left(\frac{\delta}{x}\right)^{\tilde{\alpha}}\left(\frac{1}{1+x}\right)^{\tilde{\beta}}\tag{\text{Apply Lemma \ref{lem::lawler_generalized} to $\eta^*$}}\\
&\lesssim \eps^{\alpha}\delta^{\tilde{\alpha}}x^{-\tilde{\alpha}},
\end{align*} 
where $\tilde{\beta}:=(12+2\rho-\kappa)/\kappa$.
\end{proof}

\begin{proof}[Proof of Equation (\ref{twopoints})]
Without loss of generality, we may assume that $0<\epsilon,\delta\le x/16$, $\epsilon<1/16$, and $\delta=2^{-n}$ for some integer $n$. Let $r$ be the integer such that $2^{-r}\le x< 2^{-r+1}$ and by assumption we have $n>r+3$. Define, for $k\in\N$,
\[T:=\inf\{t\ge 0: \dist(1, \eta([0,t]))\le \eps\},\quad S_k:=\inf\{t\ge 0: \dist(1+x, \eta([0,t]))\le 2^{-k}\}.\]
Note that
\[[\dist(1,\eta)\le\eps, \dist(1+x, \eta)\le\delta]=[T<\infty, S_n<\infty].\]
We decompose this event into a union of disjoint events and estimate the probability one by one:
\[[T<\infty, S_n<\infty]=[S_n<T<\infty]\cup_{k=r+3}^{n-1}[S_k<T<S_{k+1}, S_n<\infty]\cup [T<S_{r+3}<S_n<\infty].\]
\medbreak
First, by Lemma \ref{lem::rare_event}, we have
\[\PP[S_n<T<\infty]\lesssim \eps^{\alpha}\delta^{\tilde{\alpha}}x^{-\tilde{\alpha}}\lesssim \eps^{\alpha}\delta^{\alpha}x^{-\alpha}.\]
\medbreak
Second, we estimate $\PP[S_k<T<S_{k+1}, S_n<\infty]$.
\begin{align*}
\PP[S_k<T<S_{k+1}, S_n<\infty]&=\E[1_{[S_k<T<S_{k+1}<\infty]}\E[1_{[S_n<\infty]}\cond \eta([0,S_{k+1}])]]\\
&\lesssim \PP[S_k<T<S_{k+1}<\infty]\delta^{\alpha}2^{k\alpha}\tag{\text{By Equation (\ref{eqn::one_conditional_upper})}}\\
&\le \PP[S_k<T<\infty]\delta^{\alpha}2^{k\alpha}\\
&\lesssim \eps^{\alpha}\delta^{\alpha}2^{-k(\tilde{\alpha}-\alpha)}x^{-\tilde{\alpha}} \tag{\text{By Lemma \ref{lem::rare_event}}} .
\end{align*}
Summing over $k$, we have (note that $\tilde{\alpha}>\alpha$)
\[\PP[\cup_{k=r+3}^{n-1}[S_k<T<S_{k+1}, S_n<\infty]]\le \sum_{k=r+3}^{n-1} \eps^{\alpha}\delta^{\alpha}2^{-k(\tilde{\alpha}-\alpha)}x^{-\tilde{\alpha}}\lesssim \eps^{\alpha}\delta^{\alpha} x^{\tilde{\alpha}-\alpha}x^{-\tilde{\alpha}}=\eps^{\alpha}\delta^{\alpha}x^{-\alpha}.\]
\medbreak
Finally, we estimate $\PP[T<S_{r+3}<S_n<\infty]$.
\begin{align*}
\PP[T<S_{r+3}<S_n<\infty]&=\E[1_{[T<S_{r+3}<\infty]}\E[1_{[S_n<\infty]}\cond \eta([0, S_{r+3}])]]\\
&\lesssim \PP[T<S_{r+3}<\infty]\delta^{\alpha}2^{r\alpha}\tag{\text{By Equation (\ref{eqn::one_conditional_upper})}}\\
&\le \PP[T<\infty]\delta^{\alpha}2^{r\alpha}\\
&\asymp \eps^{\alpha}\delta^{\alpha}2^{r\alpha}\tag{\text{By Equation (\ref{onepoint})}}\\
&\asymp \eps^{\alpha}\delta^{\alpha}x^{-\alpha}.
\end{align*}
\end{proof}

We can also prove a conditional version of Two-Point Estimate.
\begin{lemma}\label{lem::twopoint_conditional}
Suppose that $T$ is any stopping time such that $V_T-W_T\le g_T(1)-V_T$, $\dist(1, \eta([0,T]))>16\epsilon$, and $\dist(1+x, \eta([0,T]))>16\delta$, then we have that
\[\PP[\dist(1,\eta)\le \eps, \dist(1+x, \eta)\le \delta\cond \eta([0,T])]\lesssim \left(\frac{g_T'(1)g_T'(1+x)\eps\delta}{f_T(1)(f_T(1+x)-f_T(1))}\right)^{\alpha},\]
where the constant in $\lesssim$ is uniform over $x^R\in [0^+, 1/2]$ and is independent of $T$ as long as $V_T-W_T\le g_T(1)-V_T$, $\dist(1, \eta([0,T]))>16\epsilon$, and $\dist(1+x, \eta([0,T]))>16\delta$.
\end{lemma}
\begin{proof}
Assume the same notations as in the proof of Lemma \ref{lem::onepoint_conditional}.
Given $\eta([0,T])$, the event $[\dist(1,\eta)\le \eps, \dist(1+x, \eta)\le\delta]$ implies that
\[\left[\dist(1,\tilde{\eta})\le \frac{4\eps g_T'(1)}{f_T(1)}, \quad \dist(f_T(1+x)/f_T(1),  \tilde{\eta})\le \frac{4\delta
g_T'(1+x)}{f_T(1)}\right].\]
Under the assumption $V_T-W_T\le g_T(1)-V_T$, the location of the force point of $\tilde{\eta}$ is
\[\frac{V_T-W_T}{g_T(1)-W_T}\in [0,1/2].\]
Thus, we could apply Equation (\ref{twopoints}) to $\tilde{\eta}$. Therefore, we have
\begin{eqnarray*}
\lefteqn{\PP[\dist(1,\eta)\le \eps, \dist(1+x, \eta)\le \delta\cond \eta([0,T])]}\\
&\le &\PP \left[\dist(1,\tilde{\eta})\le \frac{4\eps g_T'(1)}{f_T(1)}, \ \dist(f_T(1+x)/f_T(1), \tilde{\eta})\le \frac{4\delta g_T'(1+x)}{f_T(1)}\cond \eta([0,T])\right]\\
&\lesssim &\left(\frac{g_T'(1)g_T'(1+x)\eps\delta}{f_T(1)(f_T(1+x)-f_T(1))}\right)^{\alpha},
\end{eqnarray*}
as desired.
\end{proof} 
\section{Proof of Theorems \ref{thm::boundary_dimension} and \ref{thm::boundary_proximity}}\label{sec::boundaryproximity}
\begin{proof}[Proof of Theorem \ref{thm::boundary_dimension}]
Combining One-Point Estimate (\ref{onepoint}) and Two-Point Estimate (\ref{twopoints}) with \cite[Section 1.1]{BeffaraSLE6} and zero-one Law \cite[Lemma 3]{BeffaraSLE6} implies that the Hausdorff dimension is almost surely $1-\alpha$.
\end{proof}

In the remainder of this section, we fix $\kappa>0$ and $\rho\ge\kappa/2-2$. Suppose that $\eta$ is an $\SLE_{\kappa}(\rho)$ process with a single force point located at $x^{R}=0^+$. The proof of Theorem \ref{thm::boundary_proximity} given in this section is similar to the one in \cite[Section 3]{SchrammZhouBoundarySLE}. The major differences are that the sets considered in the present paper are all described by the Euclidean distance and all the estimates on the hitting probability are up to constants, while in \cite{SchrammZhouBoundarySLE}, the authors make use of a particular martingale to define the sets and established an explicit formula for the hitting probability of the corresponding set. We decide to skip the exactly same details as in \cite{SchrammZhouBoundarySLE}, only sketch the outline of the proof and point out the major differences.

\begin{lemma}\label{estimateofmartingale1}
Let $M_t^x$ be the local martingale defined in Lemma \ref{lem::martingale}.
Let $t>0$, $x_{0}:=\Re\eta(t)$ and $y_{0}:=\Im\eta(t)$. Suppose that $x_{0}>0$, $x>x_{0}+y_{0}$, and $\eta([0,t))$ does not intersect the vertical line segment $[x_{0}, \eta(t)]$. Then
\[M_{t}^{x}\gtrsim\frac{1}{\dist(x, \eta([0,t]))^{\alpha}}.\]
\end{lemma}
\begin{proof}
From Lemma \ref{lem::mart_dist}, we have that
\[\frac{g_t'(x)}{g_t(x)-V_t}\gtrsim \frac{1}{\dist(x,\eta([0,t]))}.\]
Thus, it is sufficient to give a lower bound for the quantity $(g_t(x)-V_t)/(g_t(x)-W_t)$.
This can be obtained by the same proof as in the proof of  \cite[Lemma 2.2]{SchrammZhouBoundarySLE}.
\end{proof}

\begin{proof}[Proof of Theorem \ref{thm::boundary_proximity}, bounded case]
The same proof as in \cite[Section 3.1]{SchrammZhouBoundarySLE} works and we need to replace the martingale by the martingale defined in Lemma \ref{lem::martingale} and replace \cite[Lemma 2.2]{SchrammZhouBoundarySLE} by Lemma \ref{estimateofmartingale1}.
\end{proof}

\begin{proof}[Proof of Theorem \ref{thm::boundary_proximity}, unbounded case]
The same proof as in \cite[Section 3.2]{SchrammZhouBoundarySLE} works with the following modifications.
\begin{itemize}
\item Define the random set $X:=\{x\ge r: \dist(x, \eta)\le h(x)\}$.
\item Replace \cite[Equation (2.9)]{SchrammZhouBoundarySLE} by One-Point Estimate (\ref{onepoint}).
\item Replace \cite[Equation (2.14)]{SchrammZhouBoundarySLE} by Two-Point Estimate (\ref{twopoints}).
\item Replace \cite[Equation (3.20)]{SchrammZhouBoundarySLE} by Lemma \ref{lem::onepoint_conditional}.
\item Replace  \cite[Equation (3.21)]{SchrammZhouBoundarySLE} by Lemma \ref{lem::twopoint_conditional}.
\end{itemize}
\end{proof}

\begin{proof}[Proof of Remark \ref{rem::estimate_hittingproba}, upper bound]
Denote $\Lambda^h_{\kappa,\rho}(x)$ by $\Lambda(x)$, and define the stopping time $\tau:=\inf\{t\ge 0: \eta(t) \text{ hits }\Gamma^h\}$. We need to evaluate $\PP[\tau<\infty]$.

Denote by $M^x_t$ the local martingale defined in Lemma \ref{lem::martingale}, and define, for $t\ge 0$,
\[Z_t:=\int_r^{\infty}\Lambda(x)M_t^x dx.\]
Then $(Z_t, t\ge 0)$ is a supermartingale with $Z_{0}=\int_{r}^{\infty}\Lambda(x)x^{-\alpha}dx$.

Given $\eta([0,\tau])$ and on the event $[\tau<\infty]$, denote $\Re{\eta(\tau)}$ (resp. $\Im{\eta(\tau)}$) by $x_0$ (resp. $y_0$),
we have that
\begin{align}
Z_{\tau}&=\int_r^{\infty}\Lambda(x)M_{\tau}^x dx\notag\\
&\ge \int_{x_0+y_0}^{2x_0}\Lambda(x)M_{\tau}^x  dx \tag{\text{Note that $y_0\le x_0/2$}}\\
&\gtrsim \Lambda(x_0)\int_{x_0+y_0}^{2x_0}M_{\tau}^x  dx\tag{\text{By Equation (\ref{condition1})}}\\
&\gtrsim \Lambda(x_0)\int_{x_0+y_0}^{2x_0}(x-x_0)^{-\alpha} dx\gtrsim 1\tag{\text{By Lemma \ref{estimateofmartingale1}}}.
\end{align}
Therefore, we obtain the upper bound:
\[Z_0\ge \E[Z_{\tau}1_{[\tau<\infty]}]\gtrsim \PP[\tau<\infty].\]
\end{proof}

\begin{proof}[Proof of Remark \ref{rem::estimate_hittingproba}, lower bound]
Denote $\Lambda^h_{\kappa,\rho}(x)$ by $\Lambda(x)$, and define the stopping time $\tau:=\inf\{t\ge 0: \eta(t) \text{ hits }\Gamma^h\}$. We need to evaluate $\PP[\tau<\infty]$.

Define
\[q:=\int_r^{\infty}\frac{\Lambda(x)}{x^{\alpha}}dx,\quad  Q:=\int_r^{\infty}\frac{\Lambda(x)}{h(x)^\alpha}1_{[\dist(x,\eta)\le h(x)]}dx.\]
By One-Point Estimate (\ref{onepoint}), we have that
$\E[Q]\asymp q$. By the same proof as in \cite[Section 3.2]{SchrammZhouBoundarySLE} (replacing \cite[Equation (2.9)]{SchrammZhouBoundarySLE} by One-Point Estimate (\ref{onepoint}) and replacing \cite[Equation (2.14)]{SchrammZhouBoundarySLE} by Two-Point Estimate (\ref{twopoints})), we have that
$\E[Q^2]\lesssim q$. By Paley-Zygmund inequality, we have that
\[q^2\asymp \E[Q]^2\le \PP[Q>0]\E[Q^2]\lesssim \PP[Q>0]q.\]
Thus,
\[\PP[Q>0]\gtrsim q.\]
The event $[Q>0]$ implies that there exists some $x>r$ such that $\dist(x,\eta)\le h(x)$. The readers can check that, it is possible to find a function $\tilde{h}$ satisfying the following conditions:
\begin{itemize}
\item $\tilde{h}$ satisfies the same assumptions for $h$;
\item There exists some $\Delta>0$ such that the integral $\int_{r+\Delta}^{\infty}(\tilde{\Lambda}(x)/x^{\alpha})dx \asymp q$, where $\tilde{\Lambda}$ is the same function as in Equation (\ref{eqn::lambda_graph}) defined for $\tilde{h}$;
\item  For any $x\ge r+\Delta$, the Euclidean balls with center $x$ and radius $\tilde{h}(x)$ are contained in the region $\Gamma^h$.
\end{itemize}
By the same proof, we have that, with probability $\gtrsim q$, there exists some $x>r+\Delta$ such that $\dist(x,\eta)\le \tilde{h}(x)$; thus, with probability $\gtrsim q$, the curve $\eta$ hits $\Gamma^h$. 
\end{proof}

\bibliographystyle{alpha}
\bibliography{hao_wu_thesis}
\bigbreak
\noindent Department of Mathematics\\
\noindent Massachusetts Institute of Technology\\
\noindent Cambridge, MA, USA\\
\noindent mengluw@math.mit.edu\\
\noindent hao.wu.proba@gmail.com

\end{document}